\documentclass[11pt, a4paper]{article}

\usepackage{amsmath, epsfig}
\usepackage{latexsym}
\usepackage{amssymb}
\usepackage{latexsym}
\usepackage{graphicx}
\usepackage{color}

\definecolor{Thistle}{rgb}{0.847,0.749,0.847}
\definecolor{Khaki}{rgb}{0.941,0.902,0.549}
\definecolor{Orchid}{rgb}{0.855,0.439,0.839}
\definecolor{MediumOrchid}{rgb}{0.729,0.333,0.827}

\definecolor{brown}{rgb}{0.8,0.5,0}
\definecolor{LightBrown}{rgb}{0.8,0.2,0.4}

\definecolor{DarkGray}{rgb}{0.78,0.78,0.78}
\definecolor{DarkMidGray}{rgb}{0.81,0.81,0.81}
\definecolor{MidGray}{rgb}{0.85,0.85,0.85}
\definecolor{LightGray}{rgb}{0.88,0.88,0.88}
\definecolor{VeryLightGray}{rgb}{0.96,0.96,0.96}

\definecolor{GrayA}{rgb}{0.7,0.7,0.7}
\definecolor{GrayB}{rgb}{0.78,0.78,0.78}
\definecolor{GrayC}{rgb}{0.80,0.80,0.80}
\definecolor{GrayD}{rgb}{0.82,0.82,0.82}
\definecolor{GrayE}{rgb}{0.84,0.84,0.84}
\definecolor{GrayF}{rgb}{0.86,0.86,0.86}
\definecolor{GrayG}{rgb}{0.88,0.88,0.88}
\definecolor{GrayH}{rgb}{0.90,0.90,0.90}
\definecolor{GrayI}{rgb}{0.92,0.92,0.92}
\definecolor{GrayJ}{rgb}{0.94,0.94,0.94}

\definecolor{VeryLightBlue}{rgb}{0.9,0.9,1}
\definecolor{LightBlue}{rgb}{0.8,0.8,1}
\definecolor{MidBlue}{rgb}{0.5,0.5,1}
\definecolor{DarkBlue}{rgb}{0,0,0.6}

\definecolor{Gold}{rgb}{1,0.843,0}

\definecolor{LightGreen}{rgb}{0.88,1,0.88}
\definecolor{MidGreen}{rgb}{0.6,1,0.6}
\definecolor{DarkGreen}{rgb}{0,0.6,0}

\definecolor{VeryLightYellow}{rgb}{1,1,0.9}
\definecolor{LightYellow}{rgb}{1,1,0.6}
\definecolor{MidYellow}{rgb}{1,1,0.5}
\definecolor{DarkYellow}{rgb}{1,1,0.2}

\definecolor{VeryLightRed}{rgb}{1,0.9,0.9}
\definecolor{LightRed}{rgb}{1,0.8,0.8}
\definecolor{MidRed}{rgb}{1,0.55,0.55}

\newtheorem{theorem}{Theorem}
\newtheorem{lemma}[theorem]{Lemma}
\newtheorem{corollary}[theorem]{Corollary}

\newtheorem{definition}{Definition}

\newtheorem{example}{Example}

\newtheorem{remark}{Remark}

\input amssym.def
\input amssym.tex

\long\def\delete#1{}

\def\qed{\hfill$\Box$\vspace{11pt}}

\def\Xi{X}

\def\PP{{\cal P}}

\def\b0{{\bf 0}}

\def\b{\beta}
\def\d{\delta}

\title{{\bf Hamiltonicity of $3$-arc graphs}}
\author{  Guangjun Xu~and~Sanming Zhou \\ \\
{\small 
Department of Mathematics and Statistics}\\
{\small The University of Melbourne}\\
{\small Parkville, VIC 3010, Australia}\\
{\small E-mail: {\it \{gx, smzhou\}@ms.unimelb.edu.au}}}

\date{\today}

\begin{document}

\maketitle

\begin{abstract}
An arc of a graph is an oriented edge and a 3-arc is a 4-tuple
$(v,u,x,y)$ of vertices such that both $(v,u,x)$ and $(u,x,y)$
are paths of length two. The 3-arc graph of a graph $G$ is 
defined to have vertices the arcs of $G$ such that two arcs
$uv, xy$ are adjacent if and only if $(v,u,x,y)$ is a
3-arc of $G$. We prove that any connected 3-arc graph 
is hamiltonian, and all iterative 3-arc graphs of any connected graph 
of minimum degree at least three are hamiltonian. 
As a corollary we obtain that any vertex-transitive 
graph which is isomorphic to the 3-arc graph of a connected arc-transitive graph of 
degree at least three must be hamiltonian. This confirms the conjecture, for this family of vertex-transitive graphs, that all vertex-transitive graphs with finitely many exceptions are hamiltonian. We also prove that if a graph with at least four vertices is Hamilton-connected, then so are its iterative 3-arc graphs.

\end{abstract}

\section{Introduction}
\label{sec:int}

A path or cycle which contains every vertex of a graph is called a {\em Hamilton path} or {\em Hamilton cycle} of the graph. A graph is {\em hamiltonian} if it contains a Hamilton cycle, and is {\em Hamilton-connected} if any two vertices are connected by a Hamilton path. The  hamiltonian problem, that of determining when a graph is hamiltonian, is a classical problem in graph theory with a long history. The reader is referred to \cite{Bondy}, \cite[Chapter 18]{BM}, \cite[Chapter 10]{Diestel} and \cite{Gould} for results on Hamiltonicity of graphs.   

In this paper we present a large family of hamiltonian graphs. Such graphs are defined by means of a graph operator, called the 3-arc graph construction, which bears some similarities with the line graph operator. This construction was first introduced in \cite{Li-Praeger-Zhou98,Zhou99} in studying a family of arc-transitive graphs whose automorphism group contains a subgroup acting imprimitively on the vertex set. (A graph is {\em arc-transitive} if its automorphism group is transitive on the set of oriented edges.) It was used in classifying or characterizing certain families of arc-transitive graphs \cite{Gardiner-Praeger-Zhou99,MPZ,Li-Praeger-Zhou98,LZ,Zhou00c,Zhou98}.  

All graphs in this paper are finite and undirected without loops. We use the term \emph{multigraph} when parallel edges are allowed. An \emph{arc} of a graph $G = (V(G), E(G))$ is an ordered pair of adjacent vertices, or equivalently an oriented edge. For adjacent vertices $u, v$ of $G$, we use $uv$ to denote the arc from $u$ to $v$, $vu$ ($\ne uv$) the arc
from $v$ to $u$, and $\{u, v\}$ the edge between $u$ and $v$. 
A \emph{3-arc} of $G$ is a 4-tuple of vertices $(v, u, x, y)$, possibly with $v = y$, such that both $(v,u,x)$ and $(u,x,y)$ are paths of $G$.   

{\bf Notation:} We follow \cite{BM} for graph-theoretic terminology and notation. The degree of a vertex $v$ in a graph $G$ is denoted by $d(v)$, and the minimum degree of $G$ is denoted by $\d(G)$.  The set of arcs of $G$ with tail $v$ is denoted by $A(v)$, and the set of arcs of $G$ is denoted by $A(G)$.

The general 3-arc construction \cite{Li-Praeger-Zhou98,Zhou99} involves a self-paired subset of the set of 3-arcs of a graph. The following definition is obtained by choosing this subset to be the set of all 3-arcs of the graph.  

\begin{definition}\label{def}
{\em 
Let $G$ be a graph. The 3-arc graph of $G$, denoted by $X(G)$, is defined to have 
vertex set $A(G)$ such that two vertices corresponding to two arcs $uv$ and $xy$ are adjacent if and only if $(v,u,x,y)$ is a $3$-arc of $G$.
}
\end{definition}

It is clear that $X(G)$ is an undirected graph with $2\,|E(G)|$ vertices 
and $\sum_{\{u,v\}\in E(G)}(d(u)-1)(d(v)-1)$ edges.
We can obtain $X(G)$ from the line graph $L(G)$ of $G$ by the following operations \cite{KZ}:
split each vertex $\{u,v\}$ of $L(G)$ into two vertices, namely $uv$ and $vu$; for any two vertices $\{u,v\}, \{x, y\}$ of $L(G)$ that are distance two
apart in $L(G)$, say, $u$ and $x$ are adjacent in $G$, join
$uv$ and $xy$ by an edge. On the other hand, the quotient
graph of $X(G)$ with respect to the partition $\PP =
\{\{uv, vu\}: \{u,v\} \in E(G)\}$ of $A(G)$ is isomorphic to the graph
obtained from the square of $L(G)$ by deleting the edges of
$L(G)$. The reader is referred to \cite{KZ,KXZ,bmg} respectively for results on the diameter and connectivity, the independence, domination and chromatic numbers, and the edge-connectivity and restricted edge-connectivity of 3-arc graphs. 

The following is the first main result in this paper.

\begin{theorem}
\label{th:ham}
Let $G$ be a graph without isolated vertices. The 3-arc graph of $G$ is hamiltonian if and only if 
\begin{itemize}
\item[\rm (a)] $\d(G) \ge 2$;
\item[\rm (b)] no two degree-two vertices of $G$ are adjacent; and 
\item[\rm (c)] the subgraph obtained from $G$ by deleting all degree-two vertices is connected.
\end{itemize}
\end{theorem}

We remark that Theorem \ref{th:ham} can not be obtained from known results on the hamiltonicity of line graphs, though $X(G)$ and $L(G)$ are closely related as mentioned above. As a matter of fact, even if $L(G)$ is hamiltonian, $X(G)$ is not necessarily hamiltonian, as witnessed by stars $K_{1,t}$ with $t\ge3$.

We define the {\em iterative 3-arc graphs} of $G$ by 
$$
X^1(G) = X(G),\;\; X^{i+1}(G) = X(X^i(G)),\;\; i \ge 1.
$$
Theorem \ref{th:ham} together with \cite[Theorem 2]{KZ} implies the following result. 

\begin{theorem}
\label{cor:ham}
\begin{itemize}
\item[\rm (a)] A 3-arc graph is hamiltonian if and only if it is connected.
\item[\rm (b)] If $G$ is a connected graph with $\d(G) \ge 3$, then $X^{i}(G)$ is hamiltonian for every integer $i \ge 1$.
\end{itemize}
\end{theorem}

We will prove Theorems \ref{th:ham} and \ref{cor:ham} in Section \ref{sec:hamil}. In Section \ref{sec:hamil-conn} we will prove the following result.

\begin{theorem}
\label{th:oddlength}
Let $G$ be a 2-edge connected graph with $\d(G) \ge 3$. If $G$ contains a path of odd length between any two distinct vertices, then its 3-arc graph is Hamilton-connected.  
\end{theorem}

A basic strategy in the proof of Theorems \ref{th:ham} and \ref{th:oddlength} is to find an Eulerian tour or an open Eulerian trail in a properly defined multigraph that produces the required Hamilton cycle or path. This is similar to the observation \cite{chart} that an Eulerian tour of a graph produces a Hamilton cycle of its line graph. 
  
Theorem \ref{th:oddlength} implies the following result. 

\begin{theorem}
\label{th:ham-c}
If a graph $G$ with at least four vertices is Hamilton-connected, then so are its iterative 3-arc graphs $X^{i}(G)$, $i \ge 1$.
\end{theorem}

Given vertex-disjoint graphs $G$ and $H$, the join $G\vee H$ of them is the graph with vertex set $V(G) \cup V(H)$ and edge set $E(G) \cup E(H) \cup \{\{u, v\}: u \in V(G), v \in V(H)\}$. Theorem \ref{th:oddlength} implies the following result. 
 
\begin{corollary}
\label{le:ham-cj}
Let $G$ and $H$ be graphs such that $\max\{\d(G), \d(H)\} \ge 2$. Then $X(G \vee H)$ is Hamilton-connected.
\end{corollary}

In the case when $G$ has a large order but small maximum degree, $X(G)$ has a large order but relatively small maximum degree. In this case the Hamiltonicity of $X(G)$ may not be derived from known sufficient conditions for Hamilton cycles such as the degree conditions in the classical Dirac's or Ore's Theorem (see \cite{Bondy,BM,Diestel,Gould}).  

In spirit, Theorems \ref{th:ham} and \ref{cor:ham} are parallel to the well-known conjecture of Thomassen \cite{Thom} which asserts that every 4-connected line graph is hamiltonian. This conjecture is still open; see \cite{CLXYZ,Gould,HTW,LLZ,Zhan05}. In contrast, Theorem \ref{th:ham} solves the hamiltonian problem for 3-arc graphs completely.  

A well-known conjecture due to Lov\'asz, formulated by Thomassen \cite{Thom1}, asserts that all connected vertex-transitive graphs, with finitely many exceptions, are hamiltonian. Since the 3-arc graph of an arc-transitive graph is vertex-transitive, Theorem \ref{cor:ham} implies the following result, which confirms this conjecture for a large family of vertex-transitive graphs. (The family of arc-transitive graphs is large from a group-theoretic point of view \cite{Praeger}.)
  
\begin{corollary}
\label{cor:vt}
If a vertex-transitive graph is isomorphic to the 3-arc graph of a connected arc-transitive graph of degree at least three, then it is hamiltonian. 
\end{corollary}

The Lov\'asz conjecture has been confirmed for several families of vertex-transitive graphs \cite{KM}, including connected vertex-transitive graphs of order $kp$, where $k \le 4$, (except for the Petersen graph and the Coxeter graph) of order $p^j$, where $j \le 4$, and of order $2p^2$, where $p$ is prime, and some families of Cayley graphs. Tools from group theory were used in the proof of almost all these results. Corollary \ref{cor:vt} has a different flavour and its proof does not rely on group theory. 

There has also been considerable interest on Hamilton-connectedness of vertex-transitive graphs. Theorem \ref{th:ham-c} implies that if a vertex-transitive graph (with at least four vertices) is Hamilton-connected, then so are its iterative 3-arc graphs. For example, it is known that every connected non-bipartite Cayley graph of degree at least three on a finite abelian group \cite{CQ} or a Hamiltonian group \cite{AQ} is Hamilton-connected. (A finite non-abelian group in which every subgroup is normal is called a Hamiltonian group.) From this and Theorem \ref{th:ham-c} we know immediately that all iterative 3-arc graphs of such a Cayley graph are also Hamilton-connected. 
\section{Preliminaries}
\label{sec:pre}
Let  $G^*$ be a multigraph. A {\em walk} in $G^*$ of length $l$ is a sequence $v_0, e_1, v_1,$ $\ldots, v_{l-1}, e_l, v_l$, whose terms are alternately vertices and edges of $G^*$ (not necessarily distinct), such that $v_{i-1}$ and $v_i$ are the end-vertices of $e_i$, $1 \le i \le l$.  
A walk is {\em closed} if its initial and terminal vertices are identical, is a {\em trail} if all its edges are distinct, and is a {\em path} if all its vertices are distinct. Often we present a trail by listing its sequence of vertices only, with the understanding that the edges used are distinct.  
A trail that traverses every edge of $G^*$ is called an {\em Eulerian trail} of $G^*$, and a closed Eulerian trail is called an {\em Eulerian tour}. A multigraph is {\em Eulerian} if it admits an Eulerian tour. It is well known that a multigraph is Eulerian if and only if all its vertices have even degrees. 
  
A {\em 2-trail} of $G^*$ is a trail of length two (and so is a path or cycle of length two). We call a 2-trail $(u,x,v)$ with mid-vertex $x$ a {\em visit to $x$} (if $u = v$, then $(u,x,u)$ is thought as entering and leaving $x$ on parallel edges). When there is no need to make distinction between $(u,x,v)$ and $(v,x,u)$, or the orientation of the visit is unknown, we write $[u,x,v]$. Two visits $(u,x,v)$ and $(u',x,v')$ are called {\em twin visits} if $\{u, v\} = \{u', v'\}$ and the four edges involved are distinct. In particular, when $u = v$, two twin visits $(u,x,u)$ and $(u,x,u)$ use four parallel edges between $u$ and $x$.  

Denote by $E^*(x)$ the set of edges of $G^*$ incident with $x \in V(G^*)$, and $d^*(x) = |E^*(x)|$ the degree of $x$ in $G^*$. In the case when $d^*(x)$ is even, a decomposition of $E^*(x)$ into a set of visits to $x$ is called a {\em visit-decomposition} of $E^*(x)$ (at $x$). In this definition the orientations of the visits in the decomposition are not important in our subsequent discussion. So we may view each visit $(u,x,v)$ in such a visit-decomposition as a non-oriented path (if $u \ne v$) or cycle (if $u = v$) of length two. As an example, if $E^*(x)=\{\{x,y\}, \{x,y\},  \{x,z\}, \{x,z\}\}$, where $\{x,y\}$ and $\{x,y\}$ are viewed as distinct edges between $x$ and $y$, then both $\{[y,x,y], [z,x,z]\}$  and  $\{[y,x,z], [y,x,z]\}$ are visit-decompositions of $E^*(x)$.

\begin{definition}
\label{def:h}
Given a visit-decomposition $J(x)$ of $E^*(x)$, define $H(x)$ to be the bipartite graph 
with vertex bipartition $\{J(x), A(x)\}$ such that $p\in J(x)$ and $xy\in A(x)$ are adjacent if and only if $y$ is not in $p$, where $A(x)$ is the set of arcs of the underlying simple graph of $G^*$ with tail $x$.  
\end{definition}

We emphasize that $H(x)$ relies on $J(x)$. One can verify the following result by using Hall's marriage theorem. 

\begin{lemma}
\label{le2}
Suppose $x$ is a vertex of $G^*$ such that $d^*(x) \geq 6$ is even and either $x$ is joined to every neighbour of $x$ by exactly two parallel edges, or $x$ is joined to one of its neighbours by exactly three parallel edges, another neighbour by a single edge, and each of the remaining neighbours by exactly two parallel edges. Let $J(x)$ be a visit-decomposition of $E^*(x)$. Then the bipartite graph $H(x)$ with respect to $J(x)$ has no perfect matchings if and only if $d^*(x)=6$ and $J(x)$ contains two twin visits.  
\end{lemma}

\begin{proof}
We have $|J(x)| = |A(x)| = d^*(x)/2$ and $\d(H(x)) \ge (d^*(x)/2)-2 \geq 1$. One can show that, if $d^*(x) \ge 8$, then the neighbourhood $N_{H(x)}(S)$ in $H(x)$ of each $S\subseteq J(x)$ has size at least $|S|$. Thus, by Hall's marriage theorem, $H(x)$ has a perfect matching when $d^*(x) \ge 8$. 

Suppose $H(x)$ has no perfect matchings, so that $d^*(x)=6$ and $|J(x)| = |A(x)| = 3$. Then there exists $S \subseteq J(x)$ such that $|N_{H(x)}(S)| < |S|$. This implies $|S|=2$ and so $|N_{H(x)}(S)| \le 1$. Denote $S = \{(u,x,v), (y,x,z)\}$, where $u, v, y,z \in N(x)$ (the neighbourhood of $x$ in $G^*$). Then $N_{H(x)}(S)=(A(x)-\{xu, xv\}) \cup  (A(x)-\{xy, xz\})= A(x)-(\{xu, xv\} \cap \{xy, xz\})$.
Since $|N(x)| = 3$ and $|N_{H(x)}(S)| \le 1$, it follows that $\{u, v\} = \{y, z\}$, and therefore $(u,x,v)$  and $(y,x,z)$  are twin visits. 
   
Conversely, if $d^*(x)=6$ and $J(x)$ contains twin visits, then $H(x)$ consists of two paths of length two and hence has no perfect matchings.\qed
\end{proof}

\begin{definition}
\label{def:decom}
Let $C: v_0, e_1, v_1, e_2, v_2, \ldots, v_{l-2}, e_{l-1}, v_{l-1}, e_l, v_l$ be an Eulerian trail of $G^*$, possibly with $v_l = v_0$. The visit $(v_{i-1}, v_i, v_{i+1})$ to $v_i$ is said to be induced by $C$, $1 \le i \le l-1$. In addition, if $C$ is an Eulerian tour, then $(v_{l-1}, v_0, v_1)$ is also a visit to $v_0$ induced by $C$.

Denote by $C(x)$ the set of visits to $x \in V(G^*)$ induced by $C$. 

Define $H_C(x)$ to be the bipartite graph at $x$ as defined in Definition \ref{def:h} with respect to the visit-decomposition $C(x)$ of $E^*(x)$. (We leave $H_C(v_0)$ and $H_C(v_l)$ undefined if $C$ is an open Eulerian trail.)
\end{definition}

Note that a vertex may be visited several times by $C$ because the vertices on $C$ may be repeated. Indeed, $C(x)$ is a visit-decomposition of $E^*(x)$ for all vertices $x$, except $v_0$ and $v_l$ when $v_0 \ne v_l$.  
 
\begin{definition}
\label{def:prec} 
Let $C$ be an Eulerian tour of $G^*$ and $J(x)$ a visit-decomposition of $E^*(x)$. We say that $C$ is compatible with $J(x)$, written  
$C(x) \prec  J(x)$, 
if for every $(a,x,b) \in J(x)$, either $(a, x, b) \in C(x)$ or $(b, x, a) \in C(x)$.
\end{definition}
 
\begin{figure}[htb]
\begin{center}
\includegraphics[width=11cm]{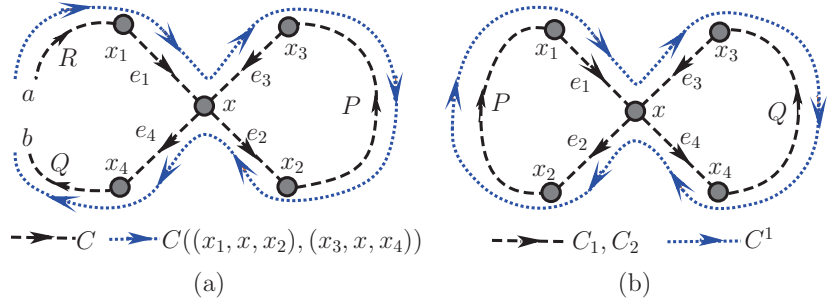}
\caption{(a)  Bow-tie operation;  (b) Concatenation operation.}
\label{fig:bt}
\end{center}
\end{figure}
 
\begin{definition}
\label{def:bowtie}
Let $C$ be a trail of $G^*$ with length at least four. Let $(x_1, x, x_2), (x_3,$  $x, x_4) \in C(x)$ be distinct visits, so that $C$ can be expressed as 
$$
C:  \overbrace{a, \ldots,     x_1}^{R}, e_1, x, e_2, \overbrace{x_2, \ldots, x_3}^{P}, e_3, x, e_4, \overbrace{x_4, \ldots, b}^{Q},
$$
possibly with $a = b$.
 
Define
$$
C((x_1, x, x_2), (x_3, x, x_4)):   \overbrace{a, \ldots,     x_1}^{R},  e_1, x, e_3^{-1}, \overbrace{x_3, \ldots, x_2}^{P^{-}}, e_2^{-1}, x, e_4, \overbrace{x_4, \ldots, b}^{Q}
$$
where $P^{-}$ is the trail obtained from $P$ by reversing its direction, and $e_2^{-1}$ and $e_3^{-1}$ are the same edges as $e_2$ and $e_3$ but with reversed orientations, respectively. (See Figure \ref{fig:bt} (a).) 

We call $C \rightarrow C((x_1, x, x_2), (x_3, x, x_4))$ the bow-tie operation on $C$ with respect to $(x_1, x, x_2)$ and $(x_3, x, x_4)$.
\end{definition}

\begin{definition}
\label{def:conca}
Let 
$$
C_1: x_1, e_1, x, e_2, \overbrace{x_2, \ldots, x_1}^{P}; \;\;\quad
C_2: x_3, e_3, x, e_4,   \overbrace{x_4, \ldots, x_3}^{Q}.
$$
be edge-disjoint closed trails of $G^*$ with $x$ as a common vertex. 
Define 
$$
C^1: x_1, e_1, x, e_3^{-1}, \overbrace{x_3, \ldots,x_4}^{Q^{-1}}, e_4^{-1}, x, e_2, \overbrace{x_2, \ldots, x_1}^{P}.
$$
We call $(C_1, C_2) \rightarrow C^1$ the concatenation operation with respect to $(C_1, C_2, (x_1, x,$ $x_2)$,  $(x_3, x, x_4))$. (See Figure \ref{fig:bt} (b).) 
\end{definition}

\begin{remark}
\label{rem:conca}
Some of $x_1, x_2, x_3, x_4$ or even all of them in Definitions \ref{def:bowtie} and \ref{def:conca} are allowed to be the same vertex. Each of $P, Q$ (and $R$ in Definition \ref{def:bowtie}) may visit some of $x, x_1, x_2, x_3, x_4$ several times, and they may have common vertices. 

In each operation above, the visits $(x_1, x, x_2)$, $(x_3, x, x_4)$ are replaced by $(x_1, x, x_3)$, $(x_4, x, x_2)$, respectively. All other visits induced by $C$ (in Definition \ref{def:bowtie}) or $C_1 \cup C_2$ (in Definition \ref{def:conca}) are retained or with orientation reversed. 

In Definition \ref{def:conca}, $C^1$ is a closed trail which covers every edge covered by $C_1$ and $C_2$. In particular, if $C_1$ and $C_2$ collectively cover all edges of $G^*$, then $C^1$ is an Eulerian tour of $G^*$. 
\end{remark}

\section{Proof of Theorems \ref{th:ham} and \ref{cor:ham}}
\label{sec:hamil}

\begin{proof}\textbf{of Theorem \ref{th:ham}}~
Denote by $S_i$ the set of vertices of $G$ with degree $i$, for $i \ge 1$. 

\medskip
Suppose  that $G$ has no isolated vertices and $X(G)$ is hamiltonian. We show that (a), (b) and (c) hold. Note first that if $G$ has a degree-one vertex, then the unique arc emanating from it gives rise to an isolated vertex of $X(G)$. Similarly, if $x, y \in S_2$ are adjacent, say, $N(x) = \{u, y\}, N(y) = \{x, v\}$, then the edge of $X(G)$ between $xu$ and $yv$ is an isolated edge no matter whether $u \ne v$ or not. Since $X(G)$ is assumed to be hamiltonian, it follows that $G$ is connected with $\d(G) \ge 2$ and $S_2$ is an independent set of $G$. 

It remains to prove that $G-S_2$ is connected. Suppose otherwise. Then we can choose a minimal subset $S$ of $S_2$ such that $G-S$ is disconnected. Note that $S \ne \emptyset$ as $G$ is connected. Let $H$ be a component of  $G-S$. The minimality of $S$ implies that each vertex of $S$ has exactly one neighbour in $V(H)$, and each vertex of $S_2$ with both neighours in $H$ (if such a vertex exists) is contained in $V(H)$.   
Denote by $A_1$ the set of arcs of $G$ with tails in $S$ and heads outside of $V(H)$.
Denote by $A_2$ the set of arcs of $G$ with tails in $V(H)$ (and heads in $V(H)$ or $S$). One can verify that the subgraph of $X(G)$ induced by $A_1 \cup A_2$ is a connected component of $X(G)$. Since there are arcs of $G$ not in $A_1 \cup A_2$, it follows that $X(G)$ is disconnected, contradicting our assumption. Hence $G-S_2$ is connected.
 
\medskip

Suppose that $G$ satisfies (a), (b) and (c). We aim to prove that $X(G)$ is hamiltonian. Note that $G$ is connected by (c). Let $G^*$ be the multigraph obtained from $G$ by doubling each edge. Then the degree $d^*(v)=2d(v)$ of each $v \in V(G)$ in $G^*$ is even. Hence $G^*$ is Eulerian.  We will prove the existence of an Eulerian tour of $G^*$ such that the corresponding bipartite graph (see Definition \ref{def:decom}) at each vertex has a perfect matching. We will then exploit such an Eulerian tour to construct a Hamilton cycle of $X(G)$. 
    
We claim first that there exists an Eulerian tour $C$ of $G^*$ such that
\begin{equation}
 \label{eq:s2}
 \mbox{if  $v \in S_2$ with $N(v)=\{u, w\}$,  then $C(v) \prec   \{(u,v,u), (w,v,w)\}$.}   
\end{equation}
To construct such an Eulerian tour, we can start from any vertex and travel as far as possible without repeating any edge such that, whenever the tour reaches a vertex of $S_2$, it returns to the previous vertex immediately.  
Since $G-S_2$ is connected, an Eulerian tour $C$ of $G^*$ satisfying (1) can be constructed this way. Note that $G^* - S_2$ is Eulerian because it is connected and all its vertices have even degrees. 
 
For an Eulerian tour $C$ of $G^*$ satisfying (\ref{eq:s2}), let $Z(C)$ denote the set of vertices $x$ such that $H_C(x)$ has no perfect matchings. Since for every $x \in S_2$, $H_C(x) \cong 2K_2$ is a perfect matching, by Lemma \ref{le2} we have $Z(C) \subseteq S_3$. 

Now we choose an Eulerian tour $C$ of $G^*$ satisfying (\ref{eq:s2}) such that $|Z(C)|$ is minimum. 
We claim that $Z(C) = \emptyset$. Suppose otherwise. Then by Lemma \ref{le2}, $C(x)$ contains twin visits for each $x \in Z(C)$. Denote $N(x)=\{x_1,x_2,x_3\}$ for a fixed $x \in Z(C)$, and assume without loss of generality that $C(x)=\{(x_1,x,x_2), (x_1,x,x_2), (x_3,x,x_3)\}$. Denote $C' = C((x_1,x,x_2), (x_3,x,x_3))$. Then $C'$ is an Eulerian tour of $G^*$ and $C'(x) = \{(x_1,x,x_2), (x_1,x,x_3), (x_2,x,x_3)\}$. One can see that $H_{C'}(x)$ is a perfect matching of three edges, and $H_{C'}(y)$ is isomorphic to $H_C(y)$ for each $y \ne x$. Thus $Z(C')$ is a proper subset of $Z(C)$, and moreover (\ref{eq:s2}) is respected by $C'$ at every $v \in S_2$. Since this contradicts the choice of $C$, we conclude that $Z(C) = \emptyset$; that is, $H_{C}(v)$ has a perfect matching for each $v \in V(G)$.   

Let $C$ be a fixed Eulerian tour of $G^*$ satisfying (\ref{eq:s2}) such that $Z(C) = \emptyset$. Let us fix a perfect matching of $H_{C}(v)$ for each $v \in V(G)$. 
Every traverse of $C$ to $v$ corresponds to a visit to $v$, say, $(u,v,w)$, and in the chosen perfect matching of  $H_{C}(v)$, $(u,v,w)$ is matched to an arc of $A(v)$ other than $vu$ and $vw$. Denote this arc by $\phi (u,v,w)$. Then for any two consecutive visits $(u,v,w), (v,w,x)$ induced by $C$ (that is, $(u, v, w, x)$ is a segment of $C$), $\phi (u,v,w)$ and $\phi (v,w,x)$ are adjacent in $X(G)$. Since $C$ is an Eulerian tour of $G^*$ and a perfect matching of each $H_{C}(v)$ is used, every arc of $G$ is of the form $\phi (u,v,w)$ for some segment $(u, v, w)$ of $C$. Therefore, if, say, $C = (u, v, w, x, y, \ldots, a, b, c, u)$, then the sequence 
$$
\phi (u,v,w), \phi (v,w,x), \phi (w,x,y), \ldots, \phi (a,b,c), \phi (b,c,u), \phi (c,u,v), \phi (u,v,w)
$$
of arcs of $G$ gives rise to a Hamilton cycle of $X(G)$. \qed
 \end{proof}

We illustrate the proof above by the following example.  
 
\begin{example}
\label{eg2}
Since the Petersen graph $PG$ (see Figure \ref{fig:pg}) satisfies the conditions in Theorem \ref{th:ham}, its 3-arc graph $X(PG)$  is  hamiltonian. Let

$$
\begin{array}{l} 
C: a_1,a_2,a_3,a_4, a_5, a_1, b_1, b_4, b_2, b_5, b_3, b_1, a_1, a_2, b_2, 
    \\ \quad \quad b_5, a_5, a_4, b_4,b_2,a_2, a_3, b_3, b_1, b_4, a_4, a_3, b_3, b_5, a_5,a_1.
   \end{array}    
$$

Then $C$ is an Eulerian tour of the multigraph $PG^*$ obtained from $PG$ by doubling each edge. One can verify that at each $a_i$ or $b_i$, $H_C(a_i)$ or $H_C(b_i)$ has a perfect matching. In $H_C(a_2)$ the `vertex' $(a_1,a_2,a_3)$ is matched to the `vertex' $a_2b_2$, and in $H_C(a_3)$, $(a_2,a_3,a_4)$ is matched to $a_3b_3$, and so on. Continuing, one can verify that $C$ gives rise to the following Hamilton cycle of $X(PG)$:
$$
\begin{array}{l} 
a_2b_2,  a_3b_3, a_4b_4,  a_5b_5,
  a_1a_2,  b_1b_3, b_4a_4,  b_2a_2, b_5a_5, 
   b_3a_3, b_1b_4,  a_1a_5, a_2a_3,  b_2b_4,  b_5b_3,  \\
 \quad a_5a_1, a_4a_3,  b_4b_1, b_2b_5,
   a_2a_1, a_3a_4,  b_3b_5, b_1a_1,  b_4b_2,
  a_4a_5,  a_3a_2, b_3b_1,  b_5b_2, a_5a_4,  a_1b_1, a_2b_2.
  \end{array}
$$
\end{example}

\begin{figure}[htb]
\begin{center}
\includegraphics[width=5cm]{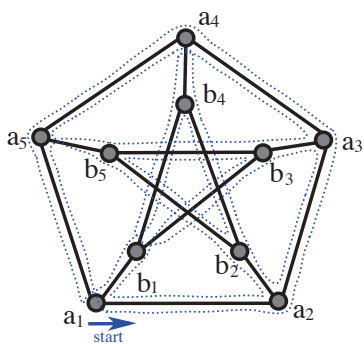}
\caption{An Eulerian tour of $PG^*$ which produces a Hamilton cycle of the 3-arc graph of the Petersen graph $PG$.}
\label{fig:pg}
\end{center}
\end{figure}

\bigskip
\begin{proof}\textbf{of Theorem \ref{cor:ham}}~
(a) Let $G$ be a graph. Define $\hat{G}$ to be the graph obtained from $G$ by replacing each degree-two vertex $v$ by a pair of nonadjacent vertices each joining to exactly one neighbour of $v$ in $G$. In \cite[Theorem 2]{KZ} it is proved that, if $\d(G) \ge 2$, then $X(G)$ is connected if and only if $\hat{G}$ is connected. One can verify that $\delta (G) \geq 2$ and $\hat{G}$ is connected if and only if (a), (b) and (c) in Theorem \ref{th:ham} hold. Thus, by Theorem \ref{th:ham}, if $X(G)$ is connected, then it is hamiltonian. The converse of this statement is obvious. 

(b) If $G$ is connected with $\d(G) \ge 3$, then $\hat{G} = G$ and so $X(G)$ is connected by \cite[Theorem 2]{KZ}. Hence, by (a), $X(G)$ is hamiltonian. Since $\d(G) \ge 3$, we have $\d(X(G)) \ge 3$. Thus, by applying (a) to $X(G)$, we see that $X^2(G)$ is hamiltonian. Continuing, by induction we can prove that $X^i(G)$ is hamiltonian for every $i \ge 1$.  \qed
\end{proof}

\section{Proof of Theorems \ref{th:oddlength} and \ref{th:ham-c}}
\label{sec:hamil-conn}

Let us first introduce an operation that will be used in the proof of Theorem \ref{th:oddlength}.
Let $G^*$ be an Eulerian multigraph and $C$ an Eulerian tour of $G^*$. Let $(z_1, x, z_2)$ be a visit of $C$ to $x$. Write    
$$
C: z_1,e_1, x,  e_2, \overbrace{z_2, \ldots, z_1}^{T}, 
$$
where $e_1$ is the oriented edge from $z_1$ to $x$, $e_2$ the oriented edge from $x$ to $z_2$, and $T$ the segment of $C$ from $z_2$ to $z_1$ covering all edges of $G^*$ except $e_1$ and $e_2$.
Add two new vertices $t, t'$ to $G^*$ and join them to $x$ by edges $e_{t}, e_{t'}$, respectively, with orientation towards $x$. Denote the resultant multigraph by $G^*_{C}(z_1, x, z_2)$. Set  
$$
W = W_{C}(z_1, x, z_2): t,e_{t}, x, e_{2},\overbrace{z_2, \ldots, z_1}^{T}, e_1, x,   e_{t'}^{-1}, t'.
$$
Since $C$ is an Eulerian tour of  $G^*$, $W$ is an open Eulerian trail of $G^*_{C}(z_1, x, z_2)$. 
Denote by $W(x)$ the set of visits to $x$ induced by $W$. As the first and last visits induced by $W$, $(t,e_{t}, x, e_{2}, z_2)$ and $(z_1,e_1, x,   e_{t'}^{-1}, t')$ are members of $W(x)$. Note that $xt,xt' \notin A(x)$. 

\begin{definition}
\label{def:kl}
Define $K_{C}(z_1, x, z_2)$ to be the bipartite graph with bipartition $\{W(x), A(x)\}$ such that an arc in $A(x)$ is adjacent to a visit $p \in W(x)$ if and only if its head does not appear in $p$. Denote by $L_{C}(z_1, x, z_2)$ the graph obtained from $K_{C}(z_1, x, z_2)$ by deleting the vertices $(t,e_{t}, x, e_{2}, z_2)$, $(z_1,e_1, x,   e_{t'}^{-1}, t')$, $xz_1$ and $xz_2$.   
\end{definition}

To prove Theorem \ref{th:oddlength}, we need to prove that, for any two distinct arcs $xy, uv$ of $G$, there exists a Hamilton path of $X(G)$ between $xy$ and $uv$. We will prove the existence of such a path by constructing a specific Eulerian trail in a certain auxiliary multigraph $G^*$. We treat the cases $x=u$ and $x\ne u$ separately in the next two lemmas. 

\begin{lemma}
\label{lem:1}
Under the condition of Theorem \ref{th:oddlength}, for any distinct arcs $xy, xv \in A(G)$ with the same tail, there exists a Hamilton path of $X(G)$ between $xy$ and $xv$. 
\end{lemma}
 
\begin{proof} 
By our assumption there exists a path in $G$ of odd length connecting $y$ and $v$. Let  
$$
P: y=x_0, x_1,x_2, \ldots, x_{l-1}, x_{l}=v 
$$   
be a path in $G$ between $y$ and $v$ with minimum possible odd length $l \ge 1$. Denote $E_0(P) = \{\{x_j,x_{j+1}\} \mid j=0,2,\ldots, l-1\}$ and $E_1(P) = \{\{x_j,x_{j+1}\} \mid j=1,3,\ldots, l-2\}$.
         
\medskip
{\bf Case 1.}  $x \not \in V(P)$. In this case let $G^*$ be obtained from $G$
by doubling each edge of $E(G)-(E(P) \cup\{\{x,y\}, \{x,v\}\})$ and tripling each edge of $E_0(P)$. 
 
\medskip    
{\bf Case 2.}  $x \in V(P)$. In this case we have $l\geq3$ and $x = x_j$ for some $1 \le j \le l-1$. If $2\le j \le l-2$, then since $l$ is odd, one of the two paths  $y, x_1, \ldots,  x_{j-1}, x, v$ and $y,  x,  x_{j+1},\ldots,  x_{l-1}, v$ would be a path of odd length connecting $y$ and $v$ that is shorter than $P$, contradicting the choice of $P$. Therefore, either $x=x_1$ or $x=x_{l-1}$. Assume without loss of generality that $x= x_1$. Define $G^*$ to be the multigraph obtained from $G$
 by doubling each edge of $E(G)-[(E(P) -\{\{x,y\}\})\cup  \{\{x,v\}\}]$  and tripling each edge of  $E_0(P)-\{\{x,y\}\}$. 

In each case above, $d^*(x) = 2d(x)-2$ and $d^*(z) = 2d(z)$ for every $z \ne x$, and hence $G^*$ is Eulerian.   

Set $a=y$ in Case 1 and $a = x_2$ in Case 2. By extending the $2$-path $a,x,v$ to an Eulerian tour, we see that there are Eulerian tours of $G^*$ which pass through $(a,x,v)$. Choose $C$ to be an Eulerian tour of $G^*$ with $(a,x,v) \in C(x)$ such that $|Z(C)|$ is minimum, where $Z(C)$ is the set of vertices $w \ne x$ of $G^*$ such that $H_{C}(w)$ has no perfect matching. 

\medskip
\textbf{Claim 1.}~$Z(C) = \emptyset$; that is, $H_{C}(w)$ has a perfect matching for every $w \ne x$.  

\medskip
\textit{Proof of Claim 1.}~
We prove this by way of contradiction. Suppose $H_C(w)$ has no perfect matching for some $w \ne x$. By Lemma \ref{le2}, $d^*(w)=6$ and $C(w)$ contains twin visits. Since $w \ne x$, we have $d(w)=3$ by the construction of $G^*$. Denote $N(w)=\{w_1, w_2, w_3\}$. In the case when each of $w_1,w_2$ and $w_3$ is joined to $w$ by two parallel edges, we apply the bow-tie operation at $w$ with respect to one of the twin visits and the third visit of $C(w)$. Similar to the proof of Theorem \ref{th:ham}, for the resultant Eulerian tour $C'$ of $G^*$, $H_{C'}(w)$ has a perfect matching, and the visit-decomposition at any other vertex is unchanged. Thus $(a,x,v) \in C'(x)$ and $Z(C')$ is a proper subset of $Z(C)$, contradicting the choice of $C$. 

It remains to consider the case where exactly one vertex of $N(w)$ is joined to $w$ by one, two or three (parallel) edges, respectively. 
Without loss of generality we may assume that there is one edge between $w_3$ and $w$, two parallel edges between $w_1$ and $w$, and three parallel edges between $w_2$ and $w$. Then $C(w)=\{[w_1,w,w_2],  [w_1,w,w_2], [w_3,w,w_2]\}$. Reversing the orientation of $C$ when necessary, we may assume $(w_1,w,w_2) \in C(w)$. Denote by $e_1, e_3$ the oriented parallel edges from $w_1$ to $w$, by $e_2, e_4, e_6$ the oriented parallel edges from $w$ to $w_2$, and by $e_5$ the oriented edge from $w$ to $w_3$. 

\medskip 
Case (a):~$C(w)=\{(w_1,w,w_2),  (w_1,w,w_2), [w_3,w,w_2]\}$. We may assume
$$
C: w_1, e_1, w, e_2, w_2, f, \ldots, g, w_1, e_3, w, e_4, w_2, h, \ldots, k, w_1.
$$
Let
$$
C': w_1, e_1, w, e_3^{-1}, w_1, g^{-1}, \ldots, f^{-1}, w_2, e_2^{-1}, w, e_4, w_2, h, \ldots, k, w_1.
$$
Then $C'$ is an Eulerian tour of $G^*$ and $C'(w)=\{(w_1,w,w_1),  (w_2,w,w_2), [w_3,w,$ $w_2]\}$. Moreover, $H_{C'}(w)$ has a perfect matching which matches $(w_1,w,w_1)$,  $(w_2,w,w_2)$, $[w_3,w,w_2]$ to $ww_2$, $ww_3$, $ww_1$ respectively. 

\medskip 
Case (b):~$C(w)=\{(w_1,w,w_2),  (w_2,w,w_1), [w_3,w,w_2]\}$. We may assume
$$
C: w_1, e_1, w, e_2, w_2, f, \ldots, g, w_2, e_4^{-1}, w, e_3^{-1}, w_1, h, \ldots, k, w_1.
$$
Denote
$$
C_1:  w_1, e_1,  w, e_3^{-1}, w_1, h, \ldots, k, w_1;\;\;\quad
C_2:  w_2, e_4^{-1}, w, e_2, w_2, f, \ldots, g, w_2.
$$
Note that each of  $C_1$ and $C_2$ is a closed trail,  
and $[w_3,w,w_2]$ is a segment of exactly one of  $C_1$ and $C_2$.

In the case when $(w_3,w,w_2) \in C(w)$ and it is in $C_2$, we first rewrite $C_2$ to highlight the position of  $(w_3,w,w_2)$ in  $C_2$:
$$
C_2': w_3,e_5^{-1},w, e_6, w_2, \ldots,w_3.  
$$
Applying the concatenation operation to $(C_1, C'_2, (w_1,  w,  w_1), (w_3,  w,  w_2))$ yields:   
 $$
C':  w_1, e_1,  w, e_5, w_3, \ldots,  w_2, e_6^{-1}, w,    e_3^{-1}, w_1, h, \ldots, k, w_1.   
$$
We have $C'(w)=\{(w_1,w,w_3),  (w_2,w,w_1), [w_2,w,w_2]\}$. Hence $H_{C'}(w)$
has a perfect matching which matches $(w_1,w,w_3)$, $(w_2,w,w_1)$, $[w_2,w,w_2]$ to $ww_2, ww_3,$ $ww_1$ respectively. 
 
In the case when $(w_3,w,w_2) \in C(w)$ and it is in $C_1$,
we first rewrite $C_1$ to highlight the position of  $[w_3,w,w_2]$ in $C_1$:  
$$
C_1': w_3,e_5^{-1},w, e_6, w_2, \ldots,w_3.  
$$
Applying the concatenation operation to $(C_2, C_1', (w_2, w, w_2), (w_3, w,  w_2))$ yields:      
$$
C': w_2, e_4^{-1}, w, e_5, w_3, \ldots,  w_2, e_6^{-1}, w,   e_2, w_2, f, \ldots, g, w_2.  
$$
Since $C'(w)=\{(w_2,w,w_3), (w_2,w,w_2), [w_1,w,w_1]\}$, $H_{C'}(w)$ has a perfect matching which matches $(w_2,w,w_3)$, $(w_2,w,w_2)$, $[w_1,w,w_1]$ to $ww_1, ww_3, ww_2$ respectively. 

The remaining case when $(w_2,w,w_3) \in C(w)$ can be dealt with similarly. 

In all possibilities above we obtain a new Eulerian tour $C'$ of $G^*$ such that $H_{C'}(w)$ has a perfect matching whilst the visit-decomposition at any other vertex is unchanged. Thus $(a,x,v) \in C'(x)$ and $Z(C')$ is a proper subset of $Z(C)$, contradicting the choice of $C$. This completes the proof of Claim 1.  

\medskip
\textbf{Claim 2.}~
There exists an Eulerian tour $C^{*}$ of $G^*$ together with a visit $(u_1, x, u_2) \in C^{*}(x)$ such that (i) $H_{C^*}(z)$ has a perfect matching for every $z \ne x$, and (ii) the bipartite graph $K_{C^*}(u_1, x, u_2)$ (as defined in Definition \ref{def:kl}) has a perfect matching under which the first and last visits induced by $W_{C^{*}}(u_1, x, u_2)$ are matched to $xy$ and $xv$ resepctively. 

Note that, for $z \ne x$, $H_{C^*}(z) = H_{W}(z)$, where $W = W_{C^{*}}(u_1, x, u_2)$. 
        
\medskip
\textit{Proof of Claim 2.}~
We will prove the existence of $C^*$ and $(u_1, x, u_2) \in C^{*}(x)$ based on $C$ as in Claim 1.  

\medskip 
Case (a):~$G^*$ was constructed in Case 1. Then $(a,x,v)=(y,x,v) \in C(x)$ and all edges of $G$ incident with $x$ except $\{x,y\}$ and $\{x,v\}$ were doubled. 

In the case when $d(x)=3$, let $z_1$ be the neighbour of $x$ in $G$ other than $y$ and $v$. One can see that $K_{C}(z_1, x, z_1)$ has a perfect matching which matches $(t,x,z_1)$,  $(y,x,v)$,  $(z_1,x,t')$ to $xy$, $xz_1$, $xv$, respectively.
       
In the case when $d(x)=4$, let $z_1$ and $z_2$ be the neighbours of $x$ in $G$ other than $y$ and $v$. Since $(y,x,v) \in C(x)$, without loss of generality we may assume $C(x)\prec \{(z_1,x,z_1), (z_2,x,z_2), (y,x,v)\}$ or $\{(z_1,x,z_2), [z_1,x,z_2], (y,x,v)\}$. If $C(x)\prec \{(z_1,x,z_1),  (z_2,x,z_2), (y,x,v)\}$, then $K_{C}(y, x, v)$ has a perfect matching which matches $(t,x,v)$,    $(z_1,x, z_1)$,  $(z_2,x, z_2)$,  $(y,x,t')$ 
to $xy$, $xz_2$, $xz_1$, $xv$, respectively. In the case when $C(x) \prec \{(z_1,x,z_2),  [z_1,x,z_2], (y,x,v)\}$, by applying the bow-tie operation at $x$ with respect to 
$((z_1,x,z_2), (y,x,v))$ we obtain a new Eulerian tour $C' = C((z_1,x,z_2), (y,x,v))$ for which 
$C'(x)=\{[z_1,x,z_2],  (z_j,x, y),   (z_{j'},x,v)\}$, where $\{j,j'\}=\{1,2\}$.
Without loss of generality we may assume $C'(x)=\{(z_1,x,z_2),  (z_j,x, y),   (z_{j'},x,v)\}$. One can see that $K_{C'}(z_1, x, z_2)$ contains a perfect matching which matches $(t,x,z_2)$, $(z_j,x, y)$, $(z_{j'},x,v)$, $(z_1,x,t')$ to $xy$, $xz_{j'}$,  $xz_j$, $xv$, respectively.
    
Assume $d(x) \geq5$. If $L_{C}(y, x, v)$ has a perfect matching, then adding the edges $\{(t, x,v), xy\}$, $\{(y,x,t'), xv\}$ to it yields a perfect matching of $K_{C}(y, x, v)$ which matches the first and last visits of $W_{C}(y, x, v)$ to $xy, xv$, respectively. 
Suppose that $L_{C}(y, x, v)$ has no perfect matchings. Similar to Lemma \ref{le2}, by using Hall's marriage theorem we can prove that $d(x)=5$ and $C(x)$ contains twin visits, say, $[z_1,x,z_2]$; that is, $C(x)\prec \{[z_1,x,z_2],  [z_1,x,z_2], [z_3,x,z_3], (y,x,v)\}$.  Without loss of generality we may assume $(z_1,x,z_2) \in C(x)$. It is not hard to see that $K_{C}(z_1, x, z_2)$ has a perfect matching  which matches  $(t,x,z_2)$, $(z_1,x,z_2)$, $[z_3,x,z_3]$, $(y,x,v)$, $(z_1,x, t')$  
to $xy$, $xz_3$, $xz_2$,  $xz_1$,  $xv$,  respectively.

\medskip
Case (b):~$G^*$  was constructed in Case 2. Then $(x_2, x, v) \in C(x)$ and all edges of $G$ incident with $x$ except $\{x,x_2\}$ and $\{x,v\}$ were doubled. 
 
In the case when $d(x)=3$, we have $C(x) \prec \{(x_2,x,v), (y,x,y)\}$ and $K_{C}(x_2, x, v)$ has a perfect matching which matches $(t, x, v), (y, x, y), (x_2, x, t')$ to $xy, xx_2, xv$, respectively. 
 
In the case when $d(x)=4$, we have $C(x)\prec \{(x_2,x,v),  [z_1,x,y], [z_1,x,y]\}$ or  
$C(x)\prec \{(x_2,x,v) (z_1,x,z_1)$, $(y,x,y)\}$, where $z_1$ is the neighbour of $x$ other than $y,v,x_2$.  
If $C(x)\prec \{(x_2,x,v),  [z_1,x,y], [z_1,x,y]\}$, let $(z_1,x,y) \in C(x)$, say. Then $K_{C}(y, x, z_1)$ has a perfect matching, namely $(t, x, z_1), (x_2,x,v), [z_1,x,y],$ $(y, x, t')$ are matched to $xy, xz_1, xx_2, xv$, respectively. 
If $C(x)\prec \{(x_2,x,v) (z_1,x,$ $z_1), (y,x,y)\}$, then $K_{C}(z_1, x, z_1)$ has a perfect matching which matches $(t, x, z_1),$   $(x_2,x,v), (y,x,y), (z_1, x, t')$ to $xy, xz_1, xx_2, xv$, respectively. 
   
Assume $d(x) \geq 5$ hereafter. In the case when $L_{C}(x_2,x,v)$ has a perfect matching, say, $M$, let $xy$ be matched to $(w_1,x,w_2)$ by $M$, where $w_1,w_2 \in N(x)-\{x_2,v,y\}$. Deleting $\{(w_1,x,w_2), xy\}$ from $M$ and then adding $\{(w_1,x,w_2),$ $xx_2\},  \{(t,x,v), xy\}$ and $\{(x_2,x,t'), xv\}$ yields a perfect matching of $K_{C}(x_2,x,v)$ satisfying (ii) in Claim 2.  

Suppose $L_{C}(x_2,x,v)$ has no perfect matchings. Similar to Lemma \ref{le2}, we can prove that $d(x)=5$ and $C(x)$  contains twin visits.
    Denote by $z_1,z_2 \ne y, v, x_2$ the other two neighbours of $x$. 
    Let  $(w_1,x,w_2)$  be one of the twin visits in $C(x)$, where $w_1, w_2 \in \{y,z_1,z_2\}$ are distinct, and let $w_3$ denote the unique vertex in $\{y,z_1,z_2\} -\{w_1, w_2\}$.
    Then $C(x)\prec \{(x_2,x,v), (w_1,x,w_2),  [w_1,x,w_2], (w_3, x,$ $w_3)\}$.  Since $w_1$ and $w_2$ are distinct, one of them, say, $w_2$, is not equal to $y$. 
Thus $K_{C} (w_1,x,w_2)$ has  a perfect matching which matches 
$(t,x,w_2)$,  $(x_2,x,v)$,   $[w_1,x,$ $w_2]$,  $(w_3, x, w_3)$,   $(w_1,x, t')$ 
to $xy$, $xw_2$, $xw_3$, $xx_2$,  $xv$, respectively.     

Since $H_C(z)$ has a perfect matching for every $z \ne x$, one can see that in all possibilities above, condition (i) in Claim 2 is satisfied by the  underlying Eulerian tour (which is $C$ or $C'$). This proves  Claim 2.

Choose an Eulerian tour $C^*: w_l, x, w_1, w_2, w_3, \ldots, w_l$ of $G^*$ together with a visit $(w_l, x, w_1) \in C^*(x)$ satisfying the conditions of Claim 2. Then $W = W_{C^*}(w_l, x, w_1): t, x, w_1, w_2, w_3, \ldots, w_{l-1}, w_l, x, t'$.
Denote by $\phi(t, x, w_1)$ ($\phi(w_l,$ $x, t')$, respectively) the arc of $G$ with tail $x$ that is matched to $(t, x, w_1)$ ($(w_l, x, t')$, respectively) by a perfect matching of $K_{C^*}(w_l, x, w_1)$ satisfying (ii) in Claim 2. Let $\phi(x, w_1, w_2)$ denote the arc matched to $(x, w_1, w_2)$ in a perfect matching of $H_{C^*}(w_1)$ ($=H_{W}(w_1)$), and let $\phi(w_1, w_2, w_3)$, $\ldots$, $\phi(w_{l-1}, w_l, x)$ be interpreted similarly. Conditions (i) and (ii) in Claim 2 ensure that 
$$
xy = \phi(t, x, w_1), \phi(x, w_1, w_2), \phi(w_1, w_2, w_3), \ldots, \phi(w_{l-1}, w_l, x), \phi(w_l, x, t') = xv
$$
is a Hamilton path of $X(G)$ connecting $xy$ and $xv$. 
\qed
\end{proof}

\begin{lemma}
\label{lem:2}
Under the condition of Theorem \ref{th:oddlength}, for distinct $xy, uv \in A(G)$ with $x \ne u$, there exists a Hamilton path of $X(G)$ between $xy$ and $uv$. 
\end{lemma}
 
\begin{proof} 
We have five possibilities to consider: $x = v$ and $y = u$; $x, y, u, v$ are pairwise distinct; $x = v$ and $y \ne u$; $y = v$ and $x \ne u$; $y = u$ and $x \ne v$. The following treatment covers all of them. 

By our assumption there exists a path of odd length connecting $x$ and $u$ in $G$. Let 
\begin{equation}
\label{eq:P}
P: x=x_0, x_1,x_2, \ldots, x_{l-1}, x_{l}=u
\end{equation}
be such a path with shortest (odd) length $l\ge1$. (It may happen that $y=x_1$ and/or $v=x_{l-1}$.) Define $G^*$ to be the multigraph obtained from $G$
 by doubling each edge of $G$ outside of $P$ and tripling each edge $\{x_j,x_{j+1}\}$ for $j=1,3,\ldots, l-2$. 
 Then $d^*(x)=2d(x)-1$,
   $d^*(u)=2d(u)-1$ and
 $d^*(z)=2d(z)$ for $z \ne x,u$.   
 
Let  $G^*_{x,u}(t,t')$ be the multigraph obtained from $G^*$ by adding  two new vertices $t, t'$   and joining them to $x,u$ respectively by a single edge. Then all vertices of $G^*_{x,u}(t,t')$ except $t$ and $t'$ have even degrees in $G^*_{x,u}(t,t')$. Hence $G^*_{x,u}(t,t')$ has Eulerian trails connecting $t$ and $t'$.    
      
Since $\delta(G)\geq 3$, we can choose $x'$ to be a neighbour of $x$ other than $y$ and $x_1$, and $u'$ a neighbour of $u$ other than $v$ and $x_{l-1}$. In addition, if $d(x)=d(u) = 3$, $y=x_1$ and $v=x_{l-1}$, say, $N(x) = \{y, x', z\}$ and $N(u) = \{v, u', w\}$, then we can choose $x'$ and $u'$ in such a way that the edges $\{x,z\}$ and $\{u,w\}$ do not form an edge cut of $G$. In fact, if $\{\{x,z\},\{u,w\}\}$  is an edge cut of $G$ in this case, then since $G$ is assumed to be 2-edge connected, $G - \{\{x,z\}, \{u,w\}\}$ has two connected components, say, $G_0$ and $G_1$ with $z,w\in V(G_0)$ and $P$ in $G_1$. Since $x'$ is in $G_1$ and removal of $\{x, x'\}$ does not disconnect $G$, one can see that $\{\{x,x'\},\{u,w\}\}$ is not an edge-cut of $G$. Thus interchanging the roles of $x'$ and $z$ produces the desired $x'$ and $u'$. (In general, at most one of $x'$ and $u'$ lies on $P$ since $P$ is a path between $x$ and $u$ with minimum odd length.) 

With $x'$ and $u'$ as above, let
$$
W': t, x, x', \overbrace{x, x_1, x_2, \ldots, x_{l-1}, u}^{P}, u',u, t', 
$$ 
where $P$ is the path given in (\ref{eq:P}). Then $W'$ is a trail of $G^*_{x,u}(t,t')$.         
Let $W$ be an Eulerian trail of $G^*_{x,u}(t,t')$ obtained by extending $W'$ to cover
all edges of  $G^*_{x,u}(t,t')$  while maintaining $(t, x, x')$ and $(u',u, t')$
as its first and last visits respectively. Such a trail $W$ exists because removing the four edges in $(t, x, x')$ and  $(u',u, t')$ from $G^*_{x,u}(t,t')$ results in a connected multigraph with $x'$ and $u'$ as the only odd-degree vertices. In addition, 
if $d(x)= 3$ and $y=x_1$, say, $N(x)=\{y, x', z\}$, since $\{\{x,z\},\{u,w\}\}$  is not an edge cut of $G$ by our choices of $x'$ and $u'$, we can choose $W$ in such a way that   $(x',x,x_1)$ is a  visit induced by $W$; similarly,  we can choose $W$  such  that   $(u',u,x_{l-1})$ is a  visit induced by $W$,
if $d(u) = 3$  and $v=x_{l-1}$, say, $N(u)=\{v, u', w\}$.   (Such a $W$ can be constructed as follows: extend $W'$ to an Eulerian trail of the multigraph obtained by deleting the parallel edges between $x$ and $z$ and/or that between $u$ and $w$, and then insert the visits $(z, x, z)$ and/or $(w, u, w)$ to this trail.) 
 In this way we obtain an Eulerian trail $W$ of $G^*_{x,u}(t,t')$ such that 
\begin{itemize}
\item[(A)] $(t, x, x')$ and $(u',u, t')$ are its first and last visits, respectively; and 
\item[(B)] if $d(x)= 3$   and $y=x_1$, say, $N(x)=\{y, x', z\}$,  then  $(x',x,x_1)\in W(x)$; and,  if $d(u) = 3$ and $v=x_{l-1}$, say, $N(u)=\{v, u', w\}$, then  $(u',u,x_{l-1})\in W(x)$. 
\end{itemize}

Similar to Claim 1, one can show that there exists an Eulerian trail of $G^*_{x,u}(t,t')$, denoted by $W$ hereafter, satisfying (A), (B) and 
\begin{itemize}
\item[(C)] $H_W(z)$ has a perfect matching for every $z \in V(G) - \{x,u\}$. 
\end{itemize}

Note that $|W(z)| = |A(z)| = d(z)$ for every $z \in V(G)$. 

\medskip
\textbf{Claim 3.}~~There exists an Eulerian trail $W^*$ of $G^*_{x,u}(t,t')$ such that 
(i) $(t, x, x')$ and $(u',u, t')$ are its first and last visits, respectively;
(ii) $H_{W^*}(x)$ has a perfect matching under which $(t, x, x')$ is matched to $xy$;  
(iii) $H_{W^*}(u)$ has a perfect matching under which $(u',u, t')$ is matched to $uv$; and 
(iv) $H_{W^*}(z)$ has a perfect matching for every $z \in V(G) - \{x,u\}$.
 
\medskip
\textit{Proof of Claim 3.}~~Let $p = (t, x, x')$ denote the first visit of $W$, and let 
        $L_W(x) = H_W(x) -\{p, xy\}$ be the subgraph of $H_W(x)$ obtained by deleting vertices $p$ and $xy$. For $S\subseteq W(x)-\{p\}$, denote by $N_{L_W(x)}(S)$ the neighbourhood of $S$ in $L_W(x)$.
       
    \medskip 
Case (a):~ $y\ne x_1$. If  $d(x)\geq5$, then $|N_{L_W(x)}(S)|\geq |S|$ for any $S$, and so $L_W(x)$ contains a perfect matching by Hall's marriage theorem.
        
 Suppose $d(x)=4$. Then $|N_{L_W(x)}(S)| \geq |S|$ for every $S$ with $|S|=1$ or $3$. Suppose $|S|=2$ and $S=\{(a,x,b),(a',x,b')\}$. Then
           $N_{L_W(x)}(S)=  [(A(x)-\{xy\})-\{xa,xb\}] \cup   [(A(x)-\{xy\})-\{xa',xb'\}] 
          =  [(A(x)-\{xy\})]- (\{xa',xb'\}\cap \{xa,xb\})$. Thus, if $|\{xa',xb'\} \cap \{xa,xb\}| \le 1$, then $|N_{L_W(x)}(S)|\geq |S|$. If  $|\{xa',xb'\}\cap \{xa,xb\}| =2$, then 
       $\{a,b\}= \{a',b'\}$ and $\{x',x_1\}\cap \{a,b\}=\emptyset$, which implies $y\in \{a,b\}$ and $|N_{L_W(x)}(S)|= |(A(x)-\{xa,xb\}| =2$. Hence $L_W(x)$ contains a perfect matching by Hall's theorem. 
        
Suppose $d(x)=3$. Then $W(x) = \{p, (x',x,y), (y,x,x_1)\}$ or $W(x)= \{p,$ $(x',x,x_1), (y,x,y)\}$. 
In the former case $L_W(x)$ clearly has a perfect matching. In the latter case, apply the bow-tie operation to $W$ with respect to $(x',x,x_1)$ and $(y,x,y)$ to obtain a new  Eulerian trail $W_0$ such that $L_{W_0}(x)$ has a perfect matching.  
       
\medskip 
Case (b):~ $y=  x_1$. Similar to Case (a), if $d(x)\geq5$, then $L_W(x)$ has a perfect matching.   
       If  $d(x)=4$,   let $N(x)=\{x', x_{1},  z_1, z_2\}$. Then $|N_{L_W(x)}(S)|\geq |S|$ 
        unless $S= \{(z_1,x,z_2), [z_1,x,z_2]\}$. In this exceptional case, $W(x)= \{p, (x',x,x_1),$ $(z_1,x,z_2), [z_1,x,z_2]\}$, and we apply the bow-tie operation 
       to $W$ with respect to $(x',x,x_1)$ and $(z_1,x,z_2)$ to obtain a new Eulerian trail $W_0$. One can show that $L_{W_0}(x)$ has a perfect matching.  
       
If  $d(x)=3$,  let $N(x)=\{x', x_{1}, z\}$. By (B), $(x',x,x_1)$ is a
visit to $x$ induced by $W$. Hence $W(x)= \{p, (x',x,x_1), (z,x,z)\}$ and $L_W(x)$ has a perfect matching. 

So far we have proved that there exists an Eulerian trail $W_1$ of $G^*_{x,u}(t,t')$ (which is either $W$ or $W_0$) satisfying (A) such that $L_{W_1}(x)$ has a perfect matching. This matching together with the edge between $(t, x, x')$ and $xy$ is a perfect matching of $H_{W_1}(x)$. Moreover, since $W$ satisfies (C), from the proof above one can see that $W_1$ satisfies (C) as well. If $H_{W_1}(u)$ has a perfect matching which matches $(u',u, t')$ to $uv$, then set $W^* = W_1$ and we are done. Otherwise, beginning with $W_1$ and using similar arguments as above, we can construct an Eulerian trail $W^*$ of $G^*_{x,u}(t,t')$ satisfying all requirements in Claim 3. This completes the  proof of Claim 3.  
       
Similar to the proof of Lemma \ref{lem:1}, we can show that the Eulerian trail $W^*$ in Claim 3 produces a Hamilton path in $X(G)$ connecting $xy$ and $uv$.  
\qed
\end{proof} 

\bigskip
\begin{proof}\textbf{of Theorem \ref{th:oddlength}}~~This follows from Lemmas \ref{lem:1} and \ref{lem:2} immediately.  
\qed
\end{proof}

In the proof of Theorem \ref{th:ham-c} we will use the following lemma which may be known in the literature. We give its proof since we are unable to allocate a reference. 

\begin{lemma}
\label{le:ham-c}
In any Hamilton-connected graph with at least four vertices, there exists a path of odd length connecting any two distinct vertices.
\end{lemma}

\begin{proof}
Let  $G$ be such a graph. Then for any distinct $u, v \in V(G)$ there exists a Hamilton path $P: u=x_0, x_1,x_2, \ldots, x_{n-1}, x_{n}=v$, where $n = |V(G)|-1$. It suffices to consider the case when $n$ is even. Denote $A=\{x_0, x_2, \ldots, x_{n}\}$ and  $B=\{x_1, x_3, \ldots, x_{n-1}\}$. Since $\{A, B\}$ is a partition of $V(G)$ and any bipartite graph other than $K_2$ is not Hamilton-connected, there exist adjacent vertices $x_i, x_j$ both in $A$ or $B$, where $j \ge i+2$. Thus $x_0, x_1, \ldots,x_{i-1}, x_{i}, x_{j}, x_{j+1},\ldots, x_{n}$ is a path of odd length between $u$ and $v$.
\qed
\end{proof}

\begin{proof}\textbf{of Theorem \ref{th:ham-c}}~~It can be verified that any Hamilton-connected graph with at least four vertices is 2-edge connected and has minimum degree at least three. Hence Theorem \ref{th:oddlength} and Lemma \ref{le:ham-c} together imply that the 3-arc graph of such a graph is Hamilton-connected (with more than four vertices). Applying this iteratively, we obtain Theorem \ref{th:ham-c}. 
\qed
\end{proof}

\section*{Acknowledgements}
Guangjun Xu was supported by the MIFRS and SFS scholarships of the University 
of Melbourne. Sanming Zhou was supported by a Future Fellowship (FT110100629) of the Australian Research Council.


\begin{thebibliography}{}
%
%


\bibitem{AQ}
B.~Alspach, Y.~S.~Qin, Hamilton-connected Cayley graphs on Hamiltonian groups, {\em Europ. J. Combin.} {\bf 22} (2001), 777--787. 

\bibitem{bmg}
C. Balbuena,  L. P.  Montejano, P.   Garc{\'{\i}}a-V{\'a}zquez,
On the connectivity and restricted edge-connectivity of 3-arc graphs,
In {\em Proceedings of   the 3rd International Workshop on Optimal Networks Topologies, 
(IWONT 2010)},  Barcelona, Spain, 9-11 June, 2010,  pp. 79--90, Barcelona, Spain, 2011. Iniciativa Digital Polit{\`e}cnica.

\bibitem{Bondy}
J.~A.~Bondy, Basic graph theory -- paths and cycles, in: Handbook of Combinatorics,Vol. I, pp. 5--110, Elsevier, Amsterdam, 1995

\bibitem{BM}
J.~A.~Bondy and U.~S.~R.~Murty, Graph Theory, Springer, New York, 2008.  


\bibitem{chart}
G.~Chartrand, On Hamiltonian line-graphs,  {\em Trans. Amer. Math. Soc.} {\bf 134} (1968),  559--566.

\bibitem{CLXYZ}
Z-H.~Chen, H-J.~Lai, L.~Xiong, H.~Yan and M.~Zhan, Hamilton-connected indices of graphs,
{\em Discrete Math.} {\bf 309} (2009), 4819--4827.

\bibitem{CQ}
C.~C.~Chen and N.~Quimpo, On strongly Hamiltonian abelian group graphs, in: Combinatorial
Mathematics VIII, K. L. McAvaney (ed), Lecture Notes in Mathematics, 884, Springer, Berlin,
1981, pp. 23--34.

\bibitem{Diestel}
R.~Diestel, Graph Theory, Springer, New York, 4th edition, 2010.


\bibitem{Gardiner-Praeger-Zhou99}
A.~Gardiner, C.~E.~Praeger and S.~Zhou, Cross-ratio graphs, {\em
J. London Math. Soc.} (2) {\bf 64} (2001), 257--272.

\bibitem{Gould}
R.~J.~Gould, Advances on the Hamiltonian problem ?a survey, {\em Graphs and Combinatorics} {\bf 19} (2003), 7--52.


\bibitem{HTW}
Z.~Hu, F.~Tian and B.~Wei, Hamilton connectivity of line graphs and claw-free graphs,
{\em J. Graph Theory} {\bf 50} (2005), 130--141.

\bibitem{MPZ}
M.~A.~Iranmanesh, C.~E.~Praeger and S.~Zhou, Finite symmetric
graphs with two-arc transitive quotients, {\em J. Combin. Theory
(Ser. B)} {\bf 94} (2005), 79--99.


\bibitem{KXZ}
M.~Knor, G.~ Xu and S.~Zhou,  A study of 3-arc graphs,
{\em Discrete Appl. Math.} {\bf 159} (2011), 344--353.

\bibitem{KZ}
M.~Knor and S.~Zhou, Diameter and connectivity of $3$-arc graphs,
{\em Discrete Math.} {\bf 310} (2010), 37--42.

\bibitem{KM}
K.~Kutnar and D.~Maru\v{s}i\v{c}, Hamilton cycles and paths in vertex-transitive graphs -- current directions, {\em Discrete Math.} {\bf 309} (2009), 5491--5500.

\bibitem{LLZ}
D.~Li, H-J.~Lai and M.~Zhan, Eulerian subgraphs and Hamilton-connected line graphs,
{\em Discrete Appl. Math.} {\bf 145} (2005), 422--428.
 
\bibitem{Li-Praeger-Zhou98}
C.~H.~Li, C.~E.~Praeger and S.~Zhou, A class of finite symmetric
graphs with 2-arc transitive quotients, {\em Math. Proc. Cambridge
Phil. Soc.} {\bf 129} (2000), 19--34.


\bibitem{LZ}
Z.~Lu and S.~Zhou, Finite symmetric graphs with $2$-arc transitive
quotients (II), {\em J. Graph Theory}, {\bf 56} (2007),
167--193.



\bibitem{Praeger}
C.~E.~Praeger, Finite symmetric graphs, in: Topics in Algebraic Graph Theory, Cambridge University Press, Cambridge, 2004, pp.179--202.

\bibitem{Thom}
C.~Thomassen, Reflections on graph theory, {\em J. Graph Theory} {\bf 10} (1986), 309--324.

\bibitem{Thom1}
C.~Thomassen, Tilings of the torus and the Klein bottle and vertex-transitive graphs on a fixed surface, {\em Trans. Amer. Math. Soc.} {\bf 323} (1991), 605--635.

\bibitem{Zhan05}
M.~Zhan, Hamiltonicity of 6-connected line graphs,
{\em Discrete Appl. Math.} {\bf 158} (2010), 1971--1975.

\bibitem{Zhou00c}
S.~Zhou, Constructing a class of symmetric graphs, {\em European
J. Combin.} {\bf 23} (2002), 741--760.

\bibitem{Zhou99}
S.~Zhou, Imprimitive symmetric graphs, $3$-arc graphs and
$1$-designs, {\em Discrete Math.} {\bf 244} (2002), 521--537.

\bibitem{Zhou98}
S.~Zhou, Almost covers of $2$-arc transitive graphs, {\em
Combinatorica} {\bf 24} (2004), 731-745. [Erratum: {\bf 27}
(2007), 745--746.]



\end{thebibliography}


\end{document}